\documentclass[11pt,a4paper]{article}
\usepackage{a4wide}
\usepackage[english]{babel}
\usepackage[utf8]{inputenc}
\usepackage[T1]{fontenc}
\usepackage{amsfonts}
\usepackage{amsmath}
\usepackage{amsthm}
\usepackage{enumitem}
\usepackage{amssymb}
\usepackage{bbm}
\usepackage{float}
\usepackage{color}
\usepackage{xfrac}
\usepackage{graphicx}
\usepackage{multicol}
\usepackage{caption}
\usepackage{subcaption}
\usepackage{mathtools}
\newcommand{\ceil}[1]{\lceil #1 \rceil}
\usepackage[font=scriptsize,labelfont=bf]{caption}

\newtheorem{obs}{Remark}[section]
\newtheorem{theorem}{Theorem}[section]

\newtheorem{corollary}{Corollary}[section]

\newtheorem{lemma}{Lemma}[section]

\def\bbE{{\mathbb E}}
\title{Scale-free behavior of networks with the copresence of preferential and uniform attachment rules.}
\author{Angelica Pachon, Laura Sacerdote \& Shuyi Yang\\
    \footnotesize{Mathematics Department ``G.~Peano'', University of Torino, Italy}}

\begin{document}
\maketitle

\begin{abstract}
Complex networks in different areas exhibit degree distributions with heavy upper tail. A preferential attachment mechanism in a growth process produces a graph with this feature. We herein investigate a variant of the simple preferential attachment model, whose modifications are interesting for two main reasons: to analyze more realistic models and to study the robustness of the scale free behavior of the degree distribution.  

We introduce and study a model which takes into account two different attachment rules: a preferential attachment mechanism (with probability $1-p$) that stresses the rich get richer system, and a uniform choice (with probability $p$) for the most recent nodes. The latter highlights a trend to select one of the last added nodes when no information is available.  The recent nodes can be either a given fixed number or a proportion ($\alpha n$) of the total number of existing nodes. In the first case, we prove that this model exhibits an asymptotically power-law degree distribution. The same result is then illustrated through simulations in the second case.  When the window of recent nodes has constant size, we herein prove that the presence of the uniform rule delays the starting time from which the asymptotic regime starts to hold.

The mean number of nodes of degree $k$ and the asymptotic degree distribution are also determined analytically. Finally, a sensitivity analysis on the parameters of the model is performed.
\end{abstract}
\medskip
		

		\medskip		
		
\section{Introduction}

Networks grow according to different paradigms and the preferential attachment mechanism is one of the simplest rules that explains some of the observed features exhibited by real networks. Barab\'asi-Albert proposed it to model the World-Wide Web (WWW) \cite{Barabasi1999}. However, it is also used for a variety of other phenomena such as citation networks or some genetic networks {\cite{342,343,Yule1925}. The preferential attachment model connects new and existing nodes  with probabilities proportional to the number of links already present. This rule is also known as rich get richer because it rewards with new links nodes with an higher number of incoming links. Networks growing according with the preferential attachment rule exhibit the \textit{scale free property} and their degree distribution are characterized by power law tails \cite{Barabasi1999,Bollobas2001,PhysRevLett.85.4633}.\\
Other models for different real world  networks request the use of different growth paradigms and do not present the scale free property. For example, some networks exhibit  the small-world phenomenon, in which sub-networks have connections between almost any two nodes within them. Furthermore, most pairs of nodes are connected by at least one short path.  On the other hand, one of the most studied models, the Erd\"os-R\'enyi random graph, does not exhibit either the power-law behavior for the degree distribution of its nodes, nor the small-world phenomenon \cite{BollobasRiordan12,ErdosRenyi59,ErdosRenyi60,janson1,KangKochPachon}. 

Mathematically, the growth of networks can be modeled through random graph processes, i.e. a family $(G^t)_{t\in \mathbb{N}}$ of random graphs (defined on a common probability space), where $t$ is interpreted as time. Different features of the model are then described as properties of the corresponding random graph process. In particular, the interest often focuses on the \textit{degree}, $\deg(v,t)$, of a vertex $v$ at time $t$, that is on the total number of incoming and (or) outgoing edges to and (or) from $v$, respectively. 
In this framework, new nodes of the Barab\'asi-Albert model  link with higher probabilities with nodes of higher degree. 

An important feature of preferential attachment models is an asymptotic power-law degree distribution:  the fraction $P(k)$ of vertices in the network  with degree $k$, goes  as $P(k)\sim k^{-\gamma}$, with $\gamma>0$, for large values of $k$.


Real world modeling instances motivated the proposal of generalizations of the Barabási-Albert model, see e.g. \cite{Berger05degreedistribution,collevecchio2013,CooperFrieze,Deijfen09,PhysRevLett.85.4629,LanskiPolitoSacerdote2014,LanskyPolitoSacerdote,12053015,RSA:RSA20137}. A common characteristic of many of these models is the presence of the same attachment rule for all the nodes of the network. However, this hypothesis is not always realistic. 

Consider, for example, a website where registered members can submit contents, such as text and posts.  Furthermore, registered users can vote previous posts to determine their position on the site's pages. Then, the submissions with the most positive votes appear on the front page together with a fixed number of the most recent posts (see www.reddit.com). This is a network in which we identify nodes with posts and links with votes. Let's consider the case in which a user submits a new post, but also votes on some previous submissions. Moreover only positive votes are allowed. It is reasonable to assume that the user tends to select and vote either the most recent posts or the most popular posts. Hence, the user votes the posts according to two different rules: with uniform probability if the user decides to select a post recently published, and  with probability proportional to the number of votes, otherwise.
This structure of network growth arises also when we consider some social networks. A new subject may connect choosing uniformly from people who recently joined the group (for instance new schoolmates) or preferring famous people present in the network since a long period of time (for example schoolmates belonging to a rock band).
A third instance arises in the case of citation networks. Typically, authors of a new paper cite recent work on the same subject as well as the most important papers on the considered topic.

Having this type of networks in our mind, in Section 2 we propose a generalization of the Barab\'asi-Albert model which takes into account the two different attachment rules for new nodes of the network. We called it  \textit{Uniform-Preferential-Attachment} model (UPA model) to pinpoint the double nature of the attachment rule.
In \cite{magner1}, Magner et al. introduced a model in which new vertices choose the nodes for their links within windows. Inspired by their idea of introducing  windows, we  formulate our model. 
We apply the preferential attachment rule to any node but we re-inforce this rule with a uniform choice for the most recent nodes. To define recent nodes we introduce windows either of fixed or linear in time amplitude. Our first result in Section 3 is the proof of recursive formulas for the expectation of the number of nodes with a given degree at a fixed time, in the case of windows of fixed amplitude. Our study  takes advantage of the existing methods  used for example in \cite{PhysRevLett.85.4633,PachonPolitoSacerdote}, together with the Azuma-Hoeffding Inequality (see Lemma 4.2.3 in \cite{janson1})}, but some of the relationships requested by these techniques are not trivial in our case. Since the mixing of uniform and preferential attachment rules suggests the possible disappearance of the scale free property we investigated this possibility. In Section 3, we also determine the degree distribution by employing a a rigorous mathematical approach in the case of fixed size windows. The use of this distribution allowed us to prove the asymptotic scale free property. These results are illustrated in Section 4 through a sensitivity analysis for the different parameters. Furthermore, in Section 4 we use simulations to study the case in which the size of the windows grows linearly in time. We show that an asymptotically  power-law degree distribution is preserved. 

Finally, Section 5 contains some concluding remarks, while the Appendix reports some auxiliary lemmas and their proofs.

\section{Uniform-Preferential-Attachment model (UPA model)}\label{UPAmodel}

In order to model instances in which recent nodes play also an important role, we propose that every new node $v_{t+1}$ selects its neighbour within a limited window $\left\{ v_{t-w+1}, ... v_{t} \right\}$ of  nodes of size $w\in\mathbb{N}$   (i.e. the $w$ youngest nodes of the network) or among all nodes $\left\{ v_{0}, ... v_{t} \right\}$. The former happens with probability $p \in \left[ 0, 1 \right]$ and the latter with probability $(1-p)$. Furthermore, the attachment rules are different in these two cases. When there is the window, the neighbour is chosen with a uniform distribution (that is, every node within the window has probability $\frac{1}{w}$ to be selected), but in absence of window, the neighbour is chosen according to a preferential attachment mechanism (that is $v_j$, $j=0,\dots,t$ has a probability proportional to $\deg(v_j,t)$ to be chosen). 

We first fix a probability $0 \le p \le 1$ and a natural number $l\ge 1$. The process starts at time $t=l$, with $l+1$  adjacent vertices,
growing monotonically by adding at each discrete time step a new vertex $v_{t+1}$ together with a directed edge connecting this with some of the vertices already present. 
As far as the window size is concerned we consider two cases:
	\begin{enumerate}
		\item For all $t\geq l$, $w:=w(t)=l$, $l\in\mathbb{N}$, that is the window size is fixed; and
		\item for each $t\geq l$, $w:=w(t)=\ceil{\alpha t}$, $0<\alpha<1$, that is the size of the window is a linear function of the size of the network.
		\end{enumerate}
	We study the first case analytically. Instead, for the second case we limit ourselves to numerical results.
The algorithmic description of the UPA model is: 

\begin{itemize}[leftmargin=20.0mm]
	\item[(a)] At the starting period $t = l$, $l\in\mathbb{N}$, the initial graph $G^l$ has $l+1$ nodes ($v_0$, $v_1$, ... $v_l$), where every node $v_j$, $1 \le j \le l$, is connected to $v_0$. 
	\item[(b)] Given $G^t$, at time $t+1$ add a new node $v_{t+1}$ together with an outgoing edge. Such edge links $v_{t+1}$ with an existing node chosen either within a window, or among all nodes present in the network at time $t$, as follows:
	\begin{itemize}
		\item with probability $p$, $v_{t+1}$  chooses its neighbour in the set $\left\{ v_{t-w+1}, ...,  v_{t} \right\}$, and each node within this window has probability  $\frac{1}{w}$ of being chosen.
		\item with probability $1-p$, the neighbour of $v_{t+1}$ is chosen from the set $\left\{ v_{0}, ..., v_{t} \right\}$, and each node $v_j$, $j=0,\dots,t$, has probability $\frac{\deg(v_j)}{2t}$ of being chosen.
	\end{itemize}
	Here, $\deg(v_j)$ indicates the total number of incoming links to $v_j$.
\end{itemize}

\begin{obs} Note that:

\begin{itemize}
	\item When $p = 0$ the UPA model reduces to the usual preferential attachment model of Barab\'asi-Albert \cite{Barabasi1999}.
	\item The initial degrees of the nodes of the UPA model,   are $\deg(v_j) = 1$ for $1 \le j \le l$ and $\deg(v_0) = l$.
\end{itemize}
\end{obs}

\section{Main results}\label{MainResults}
In this section we analyze the empirical degree distribution of a vertex in the UPA model with fixed window size, that is, for all $t\geq l$, $w(t)=l$, $l\in\mathbb{N}$.  Let us denote with $N(k,t)$, the number of vertices with degree $k$ at time $t$ in $G^t$. We prove recurrence equations for the expected degree of each node at time $t$ in Sub-Section \ref{recursiveENt} and we determine the asymptotic degree distribution in Sub-Section \ref{Adegreedistribution}. For simplicity  we will write $l$ to refer to the window size.

\subsection{Recursive formulae for $\bbE[N(k,t)]$}\label{recursiveENt}

Herein, we distinguish the case of windows of size $l=1$ from that with $l>1$. The first case is considered in Lemma \ref{propA1} while the second is the subject of Lemma \ref{Ewindow}. 

\begin{lemma}\label{propA1}
	In the UPA model with fixed windows size $l=1$, it holds:
\begin{equation}
	\label{l1theequations}
	\bbE{N}(k,t+1)  =
	\begin{cases}
	 \left( 1-\frac{1-p}{2t} \right)\bbE[{N}(1,t)] + (1-p),&\text{ if } k=1 \\
	 \left( 1-\frac{1-p}{t} \right)\bbE[{N}(2,t)] + \frac{1-p}{2t}\bbE[{N}(1,t)] + p, &\text{ if } k=2\\
	 \left( 1-\frac{(1-p)k}{2t} \right)\bbE[{N}(k,t)] + \frac{(1-p)(k-1)}{2t}\bbE[{N}(k-1,t)], &\text{ if } k>2, \\
		\end{cases}
	\end{equation}
	with initial conditions $\bbE{N}(k,1)=2$ for $k=1$ and $\bbE{N}(k,1)=0$ otherwise. 
\end{lemma}
\begin{proof}[Proof of Lemma \ref{propA1}]
We start calculating the probability that a node with degree $k$ at time $t$ receives a new link from $v_{t+1}$. To achieve this we distinguish  two cases, $k > 1$ and $k = 1$. Let us consider a node with degre $k>1$. By construction its degree can increase only when it is  selected through a preferential attachment mechanism (which happens with probability $1-p$). On the contrary when $k=1$  we should also consider the effect of the window (which happens with probability $p$).
Thus, the probability that the new link exiting from $v_{t+1}$ attains a node with degree $k>1$ at time $t$ is given by
\begin{equation}\label{probkreceivelink}
(1-p) \frac{N(k,t)k}{\sum_{k'=1}^{+\infty}N(k',t)k'}=(1-p) \frac{N(k,t)k}{2t},
\end{equation}
while  the probability that the new link attains a node with degree $k=1$ at time $t$ is
\begin{equation}
(1-p) \frac{N(1,t)}{2t} + p .
\end{equation}

Let $\mathcal{G}_{s}$ be the $\sigma$-field generated by the appearance of edges up to time $s$. Observe that conditioning on $\mathcal{G}_t$, $N(k,t+1)$, $k\geq 1$, is a random variable depending on $N(k,t)$. Thus, if $k=1$ at time $t+1$ the number of nodes of degree 1, $N(1,t+1)$, remains unchanged or increases by 1. The first possibility happens either in presence of a window, or when $v_{t+1}$ is attached to an existent node of degree 1 at time $t$ in absence of window. The second possibility arises if the selected node has degree larger than 1. That is,
 
\begin{equation}\label{P1}
\begin{split}
N(1,t+1) = 
\begin{cases}
N(1,t) & \text{ ~ w.p. ~ } p+\frac{(1-p)N(1,t)}{2t}\\
N(1,t)+1 & \text{ ~ w.p. ~ } \frac{(1-p)(2t-N(1,t))}{2t}
\end{cases}
\end{split}
\end{equation}

where the abbreviation w.p. means ``with probability''. 

Similarly, conditioned on $\mathcal{G}_t$  the probability distribution of $N(k,t+1)$, $k\geq2$, is

\begin{equation}\label{P2}
\begin{split}
N(2,t+1) = 
\begin{cases}
N(2,t)+1 & \text{ ~ w.p. ~ } p+\frac{(1-p)N(1,t)}{2t}\\
N(2,t)-1 & \text{ ~ w.p. ~ } \frac{(1-p) \cdot 2N(2,t)}{2t}\\
N(2,t) & \text{ ~ w.p. ~ } \frac{(1-p)\left[ 2t-N(1,t)-2N(2,t) \right]}{2t},
\end{cases}
\end{split}
\end{equation}
and for $k > 2$
\begin{equation}\label{Pk}
\begin{split}
N(k,t+1) = 
\begin{cases}
N(k,t)+1 & \text{ ~ w.p. ~ } \frac{(1-p)N(k-1,t)(k-1)}{2t}\\
N(k,t)-1 & \text{ ~ w.p. ~ } \frac{(1-p) N(k,t) k}{2t}\\
N(k,t) & \text{ ~ w.p. ~ } p+\frac{(1-p)\left[ 2t-N(k,t)k-N(k-1,t)(k-1) \right]}{2t}.
\end{cases}
\end{split}
\end{equation}
Taking the conditional expectations given $\mathcal{G}_t$  of (\ref{P1}), (\ref{P2}) and (\ref{Pk}), respectively, we get
\begin{equation}\label{E1}
\begin{split}
\bbE\left[ N(k,t+1) \mid \mathcal{G}_t\right]=
\begin{cases}
\left( 1-\frac{1-p}{2t} \right)N(1,t) + (1-p), &\text{ if } k=1\\
\left( 1-\frac{1-p}{t} \right)N(2,t) + \frac{1-p}{2t}N(1,t) + p, &\text{ if } k=2\\
\left( 1-\frac{(1-p)k}{2t} \right)N(k,t) + \frac{(1-p)(k-1)}{2t}N(k-1,t),&\text{ if }  k>2.
\end{cases}
\end{split}
\end{equation}
Finally, taking expectation of both sides of (\ref{E1}) we obtain the desired result.
\end{proof}

To switch to the case of window size $l>1$, let us define $M_m(t)$ as the degree of the $m$-th node from the left, inside the window in $G^t$. That is, $M_1(t)=d(v_{t-l+1}),M_2(t)=d(v_{t-l+2}),\dots,M_m(t)=d(v_{t-l+m}),\dots,M_l(t)=d(v_{t})$ (see equations (\ref{PM1}) and (\ref{PMk}) for recursive formulae for $M_m(t)$).
Then, it holds:
\begin{lemma}\label{Ewindow}
In the UPA model with fixed windows size $l>1$:
	\begin{equation}\label{theequations-l-general-1}
	\bbE[N(k,t+1)] = 
	\begin{cases}
	\bbE[N(1,t)] + 1 - \frac{p\sum_{m=1}^{l} \mathbb{P}(M_m(t)=1) }{l} -(1-p)  \frac{\bbE[N(1,t)]}{2t},&\text{ if } k=1\\
	 \\
	\bbE[N(k,t)] + \frac{p}{l}\left[ \sum_{m=1}^{l} \mathbb{P}(M_m(t)=k-1) - \sum_{m=1}^{l} \mathbb{P}(M_m(t)=k) \right] \\
	+\frac{1-p}{2t} \left\{ (k-1)\bbE[N(k-1,t)] - k\bbE[N(k,t)] \right\}, &\text{ if } k>1
	\\
	\end{cases}
	\end{equation}
	with initial conditions $\bbE{N}(k,l)=l$ for $k=1$, $\bbE{N}(k,l)=1$ for $k=l$, and $\bbE{N}(k,l)=0$ otherwise.
\end{lemma}

\begin{proof}[Proof of Lemma \ref{Ewindow}]
We proceed in the same way as in the proof of Lemma \ref{propA1}. Recall that at time $t+1$,
$v_{t+1}$ appears with degree 1, and the number of nodes with degree 1 remains unchanged if and only if the neighbour chosen by $v_{t+1}$ is a vertex with degree 1 too. This happens  with probability $ N(1,t)/2t$ if there is no window, and with probability $\sum_{m=1}^{l}\mathbbm{1}(M_m(t)=1)/l$ if the selected node lies inside the window.
Thus, the random variable $N(1,t+1)$ given $\mathcal{G}_t$ is given by
\begin{equation}\label{Nl1}
N(1,t+1) = \begin{cases}
N(1,t) & \text{ ~ w.p. ~ } p\frac{\sum_{m=1}^{l}\mathbbm{1}(M_m(t)=1)}{l} + (1-p)\frac{1 \cdot N(1,t)}{2t}\\
N(1,t)+1 & \text{ ~ w.p. ~ } 1-\left[p\frac{\sum_{m=1}^{l}\mathbbm{1}(M_m(t)=1)}{l} + (1-p)\frac{1 \cdot N(1,t)}{2t} \right] .\\
\end{cases}
\end{equation}

With a similar reasoning, for $k \ge 2$, we obtain  $N(k,t+1)$ given $\mathcal{G}_t$, 
\begin{equation}\label{Nlk}
N(k,t+1) = \begin{cases}
N(k,t)-1 & \text{ ~ w.p. ~ } p\frac{\sum_{m=1}^{l}\mathbbm{1}(M_m(t)=k)}{l}+(1-p)\frac{kN(k,t)}{2t}\\
N(k,t)+1 & \text{ ~ w.p. ~ } p\frac{\sum_{m=1}^{l}\mathbbm{1}(M_m(t)=k-1)}{l}+(1-p)\frac{(k-1)N(k-1,t)}{2t}\\
N(k,t) & \text{ ~ w.p. ~ } 1 - p\frac{\sum_{m=1}^{l}\mathbbm{1}(M_m(t)=k)}{l}-(1-p)\frac{kN(k,t)}{2t} \\
~ & ~ ~ ~ ~ ~ ~ ~ -p\frac{\sum_{m=1}^{l}\mathbbm{1}(M_m(t)=k-1)}{l}-(1-p)\frac{(k-1)N(k-1,t)}{2t}.
\end{cases}
\end{equation}

Taking the conditional expectation of (\ref{Nl1}) and (\ref{Nlk}) we get
\begin{equation}\label{13}
\bbE(N(1,t+1) \mid \mathcal{G}_t) = N(1,t) + 1 - \frac{p\sum_{m=1}^{l} \mathbbm{1}(M_m(t)=1) }{l} -(1-p)\frac{N(1,t)}{2t},
\end{equation}
and
\begin{equation}\label{14}
\begin{split}
\bbE[N(k,t+1) \mid \mathcal{G}_t] = N(k,t) + & \frac{p}{l}\left[ \sum_{m=1}^{l} \mathbbm{1}(M_m(t)=k-1) - \sum_{m=1}^{l} \mathbbm{1}(M_m(t)=k) \right] \\
+ & \frac{1-p}{2t} \left[ (k-1)N(k-1,t) - kN(k,t) \right] .
\end{split}
\end{equation}
Finally, taking the  expectation of both sides of (\ref{13}) and (\ref{14}) we obtain (\ref{theequations-l-general-1}).
\end{proof}

\subsection{Asymptotic degree distribution}\label{Adegreedistribution}
\begin{theorem}\label{main1}
	Consider the UPA model with fixed windows size $l \in \mathbb{N}$.
	Then,
	\begin{equation}
	\frac{N(k,t)}{t} \rightarrow  P(k)
	\end{equation} \\
	in probability as $t\rightarrow\infty$.
	
	Furthermore, for $l=1$ it holds:
\begin{eqnarray}\label{pkl1}
{P}(k)=\begin{cases}
\frac{2(1-p)}{3-p} & \text{ if $k=1$}\\
\frac{(1-p)^2}{(2-p)(3-p)} +\frac{p}{2-p} & \text{ if $k=2$}\\
\Big(\frac{2}{1-p}+2\Big)\Big(\frac{2}{1-p}+1\Big)B\left(k,1+\frac{2}{1-p}\right)\overline P(2) & \text{ if $k>2$},
\end{cases}
\end{eqnarray}
while for $l>1$ we have:
\begin{eqnarray}\label{pkll1}
{P}(k)=\begin{cases}
\frac{2}{(3-p)}\left(1-\frac{p}{l}\right)^l& \text{ ik $k=1$}\\ 
\frac{2}{2+k(1-p)}\left(\frac{p}{l}(H_{k-1}-H_k)+\frac{(1-p)(k-1)}{2}{P}(k-1)\right)& \text{ if $k=2,\dots,l+1$}\\
\frac{B\left(k,l+2+\frac{2}{1-p}\right)}{B\left(l+1,k+1+\frac{2}{1-p}\right)}{P}(l+1)& \text{ if $k>l+1$},
\end{cases}
\end{eqnarray}
where $B(x,y)$ is the Beta function and 
\begin{equation}\label{FHk}
H_k=\begin{cases}
	\left(\frac{p}{l}\right)^{k-1}\sum_{m=1}^{l-(k-1)}\binom{l-m}{l-m-(k-1)}\left(1-\frac{p}{l}\right)^{l-m-(k-1)} &\text{ if } k=1,\dots,l.\\
	0	&\text{ if } k>l.
	\end{cases}
	\end{equation} 
\end{theorem}
\begin{proof}
The proof includes the following steps:
\begin{enumerate}
	\item[(1)] we determine recursivelly $\bbE[N(k,t)], k=1,2,\ldots$;
	\item[(2)] we prove the existence of ${P}(k) := \lim_{t \to \infty}\bbE[N(t,k)]/t$,
	\item[(3)] we determine an explicit expression for ${P}(k)$,
	\item[(4)] we use the Azuma-Hoeffding Inequality (see Lemma 4.2.3 in \cite{janson1}) to prove convergence in probability of $N(k,t)/t$ to ${P}(k)$.  
\end{enumerate} 
The first step is solved by the results of Sub-Section 3.1. For the proof of steps (2) and (3) it is necessary to distinguish the cases of windows size $l=1$ and $l > 1$, respectively. The proof of the second and third steps request some technical results that are reported and proved in the Appendix. In particular, Lemma \ref{lemmaconvk} and Lemma \ref{l-lemmaconv1} in Appendix prove the existence of ${P}(k)$, in the cases of windows size $l=1$ and $l > 1$, respectively. Furthermore, Lemma \ref{l1asppl} and Lemma \ref{5.1} in Appendix give the explicit expression of such limits in the cases of windows size $l=1$ and $l > 1$, respectively.
It remains to prove step (4).

Define $X_s = E\left( N(k,t) \mid \mathcal{G}_s \right)$, $s \le t$, and observe that $X_s$ is a martingale. Furthermore, $\mid X_s-X_{s-1} \mid \le 2$. This happens because the degree of any vertex $v_k, k\neq i,j$ is not affected by the rising of $v_{s-1}$ and $v_s$, whether these last vertices link to $v_i$ and $v_j$, $i<s-1$ and  $j<s$, respectively. Furthermore, they do not affect the probabilities of the vertices $v_k$ that will be chosen later.  Therefore, applying the Azuma-Hoeffding inequality with
$X_t=N(k,t)$, $X_0=\bbE[N(k,t)]$ and taking $x = \sqrt{t\log(t)}$,  we obtain
\begin{equation*}
P\left(\mid N(k,t)/t - \bbE[N(k,t)]/t\mid > \sqrt{log(t)/t} \right) \le t^{-\frac{1}{8}}.
\end{equation*}
Hence, as $t\rightarrow\infty$ we have $N(k,t)/t\rightarrow {P}(k)$ in probability. 
\end{proof}
Recalling that $B(x,y)=\Gamma(x)\Gamma(y)/\Gamma(x+y)$ and using that $\Gamma(x+a)/\Gamma(x+b)\sim x^{a-b}[1+\sfrac{(a-b)(a+b-1)}{2x}]$, for $x$ large enough (see \cite{Abramowitz}),  we get the following. 
\begin{corollary}\label{remarkTeo}
As $k\rightarrow\infty$, for $l=1$ 
\begin{equation}\label{asymDegree}
\frac{N(k,t)}{t}\sim C_p \Big[k^{-\left(1+\frac{2}{1-p}\right)}-\frac{3-p}{(1-p)^2} k^{-\left(2+\frac{2}{1-p}\right)}\Big],
\end{equation}
where $C_p=\Gamma(1+\sfrac{2}{(1-p)})(\sfrac{2}{(1-p)}+2)(\sfrac{2}{(1-p)}+1) P(2)$,
and for $l>1$, 
\begin{equation}\label{asymDegree2}
\frac{N(k,t)}{t}\sim C_{p,l} \Big[k^{-\left(1+\frac{2}{1-p}\right)}-\frac{3-p}{(1-p)^2} k^{-\left(2+\frac{2}{1-p}\right)}\Big],
\end{equation}
where $C_{p,l}=\Gamma(l+2+\sfrac{2}{(1-p)})\big(\Gamma(l+1)\big)^{-1} P(l+1)$.
\end{corollary}
Note that the value of $p$ is part of the exponent of the power law in (\ref{asymDegree}) and (\ref{asymDegree2}). Moreover, when the UPA model corresponds to the Barab\'asi-Albert model, that is when $l=1$ and $p=0$, the power law exponent is equal to $-3$, as expected.


\section{Examples of the UPA model and its generalization}

In this section we show through some examples, our analytical and asymptotic results. We then study the UPA model when the windows size grows linearly in time. In this case, we use simulations to study the empirical degree distribution. We divided this section into three Sub-sections. In Sub-Section \ref{ML} we study the mean number of nodes having a certain degree in different instances, while in Sub-Section \ref{FWS} we illustrate our results on the empirical degree distribution. Finally, Sub-Section \ref{LWS} illustrates the case with time dependent windows size.

\subsection{Mean number of nodes}\label{ML}
From Corollary \ref{remarkTeo} we know that the presence of uniform attachment is not sufficient to asymptotically destroy the scale free feature. Herein we investigate the role of the presence of the windows for fixed times and using different weights for the two types of attachment rule. We use Lemmas \ref{propA1} and \ref{Ewindow} with formulas (\ref{PM1}) and (\ref{PMk}) to compute  $\bbE\left[N(k,t)\right]$ as a function of $k$ and $t$ for different values of $l$ and of $p$.  
In Fig. 1a, we first fix $t=101$ and $p=0.5$, to study the mean number of nodes having degree $k$, with $k=1,2,\cdots$, for different sizes of the window. Since the grade $k=1$ is forced by the model hypothesis ($N(1,l)=l$), the figure is of interest for $k \geq 2$. As expected, the effect of the size of the window is stronger for lower degrees, and it tends to disappear as the degree of the node increases. The same result is observed increasing the considered time (see inset of Fig. 1a). Fig. 1b illustrates the increase of $\bbE\left[N(10,t)\right]$ as $t$ increases, for different sizes of the window. 
First of all, we can notice that the size of the window penalizes the growth of $\bbE\left[N(10,t)\right]$. Each of the curves of $\bbE\left[N(10,t)\right]$, corresponding to $l=10$ and $l=100$, crosses the analogous curve for $l=1$ after a certain time. After such times, the curves corresponding to $l=10$ and $l=100$ grow faster than in the case of $l=1$. This fact can be explained by noting that for $l=1$, the uniform rule does not privilege the choice of a vertex with degree 10. In this case, the uniform rule determines the choice of the last added vertex, which has degree 1. This also determines the faster increase of $\bbE\left[N(10,t)\right]$ for $l=10$ than for $l=100$, as $t$ increases enough. This phenomenon does not disappear if we increase the value of $k$ (see inset of Fig. 1b). The shapes of the curves do not change but there is a remarkable change of scale.
\begin{figure}[H]
	\centering
	\begin{subfigure}{.48\textwidth}
		\centering
		\includegraphics[width=1\linewidth]{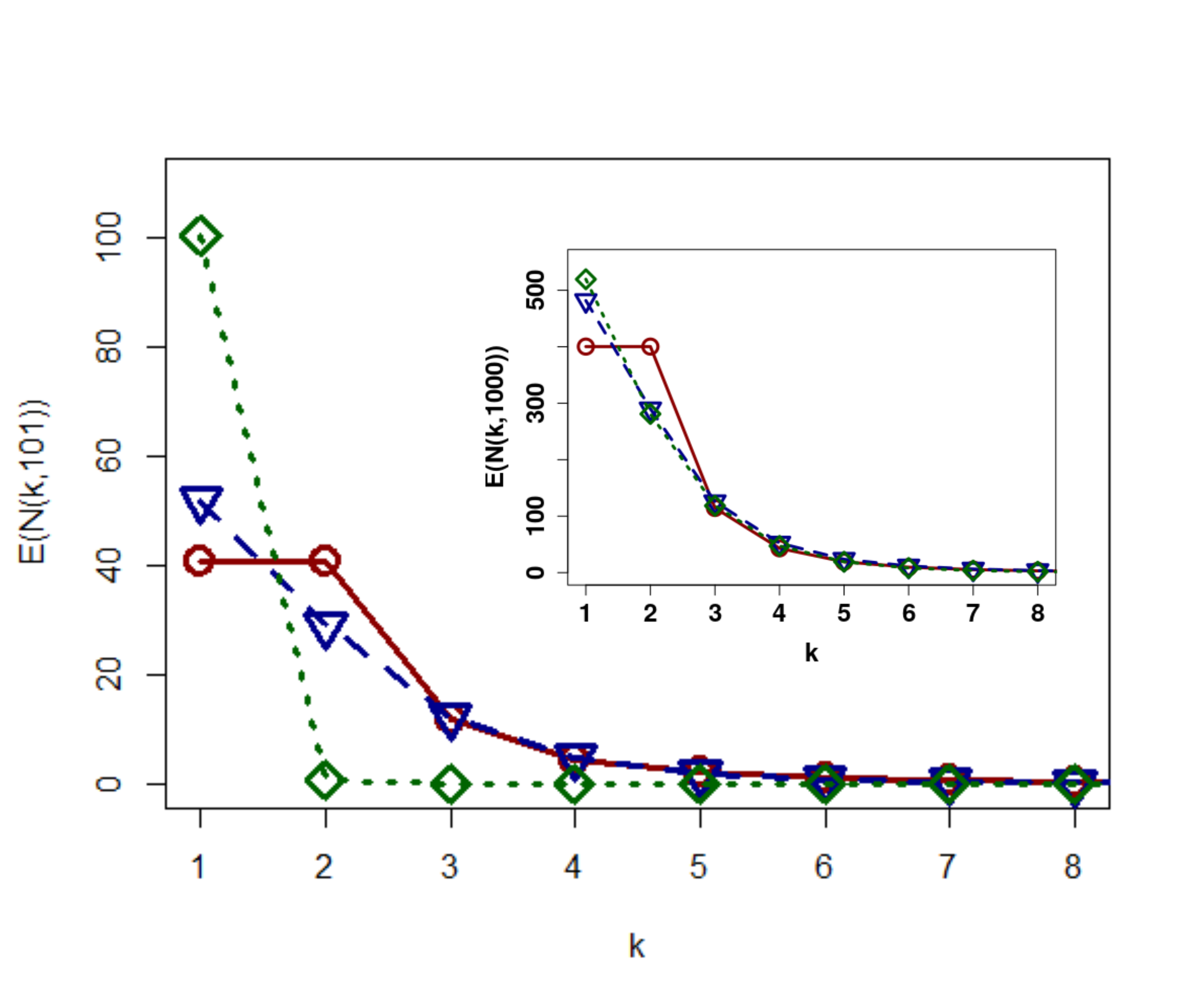}
	\end{subfigure}
	~ 
	\begin{subfigure}{.48\textwidth}
		\centering
		\includegraphics[width=1\linewidth]{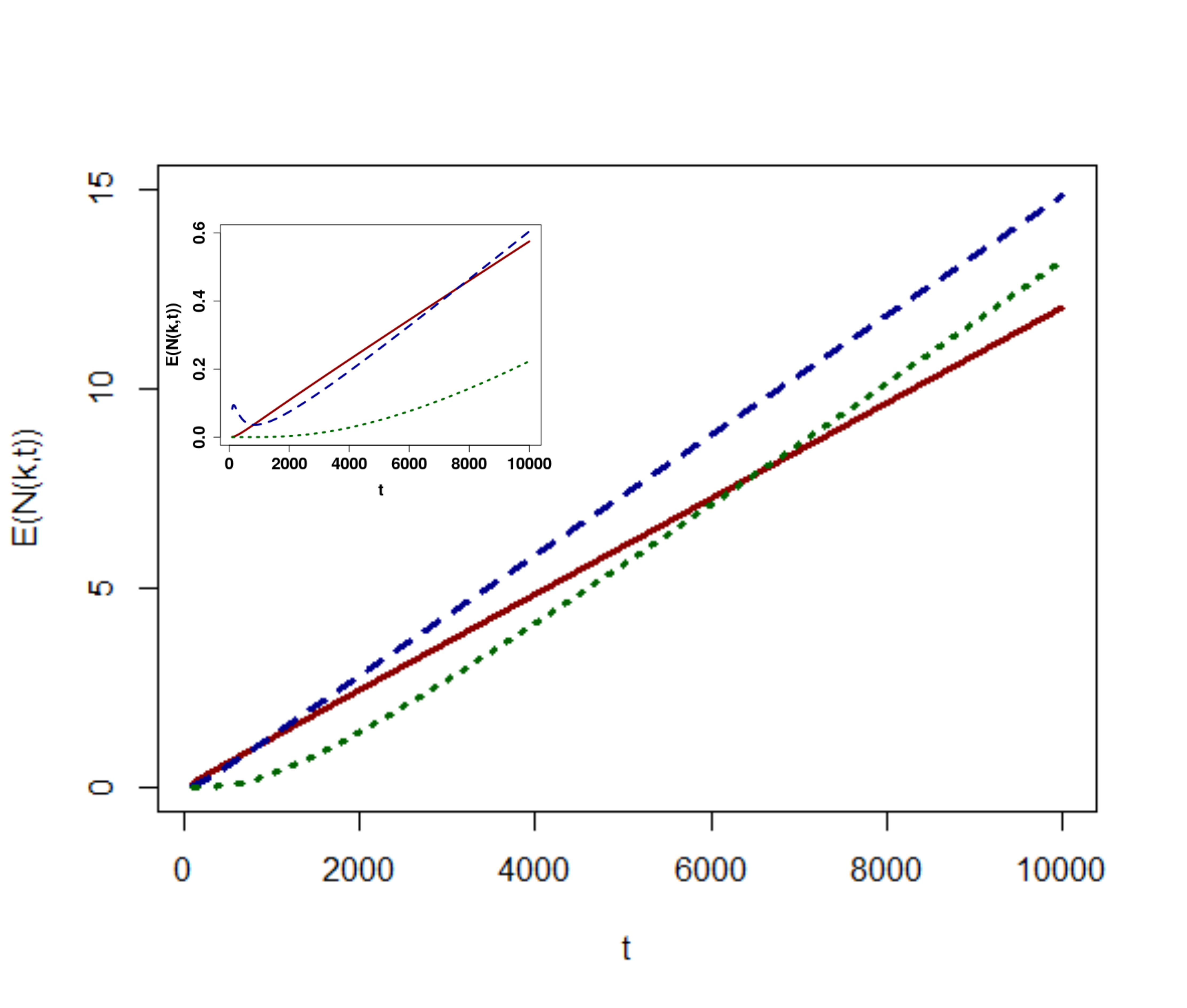}
	\end{subfigure}
\caption*{Fig.1a (left panel): The expected number of vertices with degree $k$  versus $k$ are shown, for $t=101$, $p=0.5$, and different sizes of the window, $l=1$ (red, solid line $\circ$), $l=10$ (blue, dashed line $\triangledown$) and $l=100$ (green, dotted line $\lozenge$). \textit{Inset:} the same picture but here $t=1000$.
Fig.1b (right panel): The expected number of vertices with degree $k$  versus $t$ are shown, for $k=10$, $p=0.5$, and different sizes of the window, $l=1$ (red, solid line), $l=10$ (blue, dashed line) and $l=100$ (green, dotted line). \textit{Inset:} the same picture but here $k=20$.}
\end{figure}

As far as the role of $p$ is concerned, we note that it only affects the nodes with low degree (see Fig. 2a). When $p$ is large, the uniform attachment mechanism helps links to nodes of lower degree but its role tends to disappear for nodes of higher degree. Again, the shape of the curves remains the same if we change $t$ but the scale changes (see Fig. 2a and its inset).
Fig. 2b illustrates $\bbE\left[N(10,t)\right]$ as a function of $t$ for different values of $p$. We observe that for small values of $t$, the curves of $\bbE\left[N(10,t)\right]$ decrease, and after some time they start to increase. That is just an effect of the initial condition, $\bbE\left[N(10,10)\right]=1$. At the beginning of the process there are few vertices, and only one with degree 10, $v_0$. Then, $v_0$ can easily increase its degree by receiving a new link.  A stronger presence of uniform attachment rule (greater $p$) decreases the speed of the growth of the degree of the nodes. Note that for $p=1$, pure uniform attachment, $\bbE\left[N(10,t)\right]$ remains constant when $t$ grows. For the case $k=20$ (see inset of Fig. 2b), the initial condition is $\bbE\left[N(20,10)\right]=0$. Here we observe that for small values of $t$, the curves of $\bbE\left[N(20,t)\right]$ increase reaching a local maximum, then they decrease and after some time they start to increase again. This behavior of the curves for small values of $t$ is a consequence of the zero value of the initial condition, and of the fact that $l=10$. In this range of $t$, $\bbE\left[N(20,t)\right]$ is only affected by the preferential attachment rule.

Fig. 1a, 1b, 2a and 2b suggest a role of the two attachment rules for the growth rate of low degree nodes. During the initial development of the network, lower degree nodes receive more links thanks to the presence of the window. However, the considered examples suggest a marginal asymptotic role of the uniform attachment in making significant changes to long term dynamics.

\begin{figure}[H]
	\centering
	\begin{subfigure}{.48\textwidth}
		\centering
		\includegraphics[width=1\linewidth]{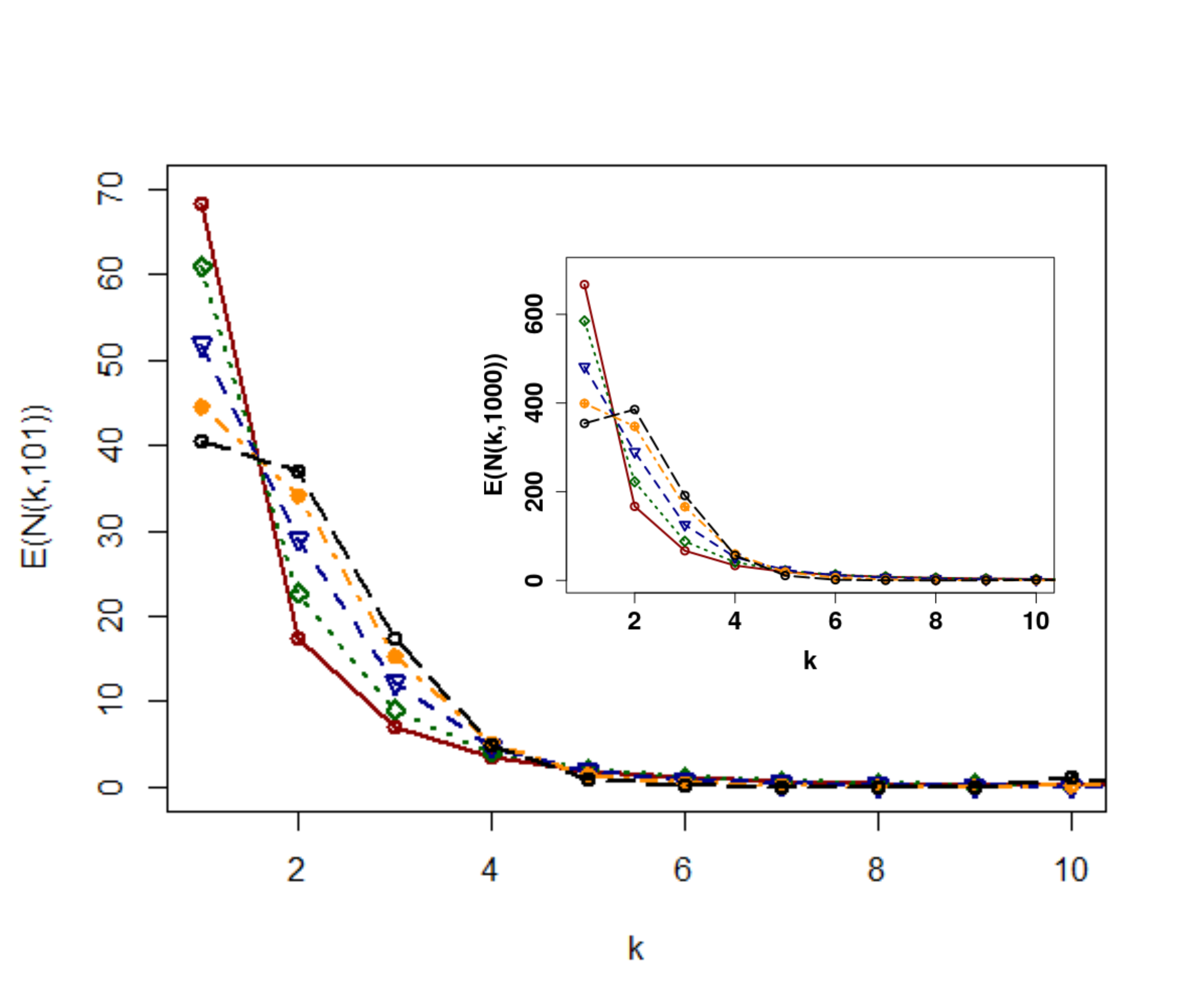}
	\end{subfigure}
	~ 
	\begin{subfigure}{.48\textwidth}
		\centering
		\includegraphics[width=1\linewidth]{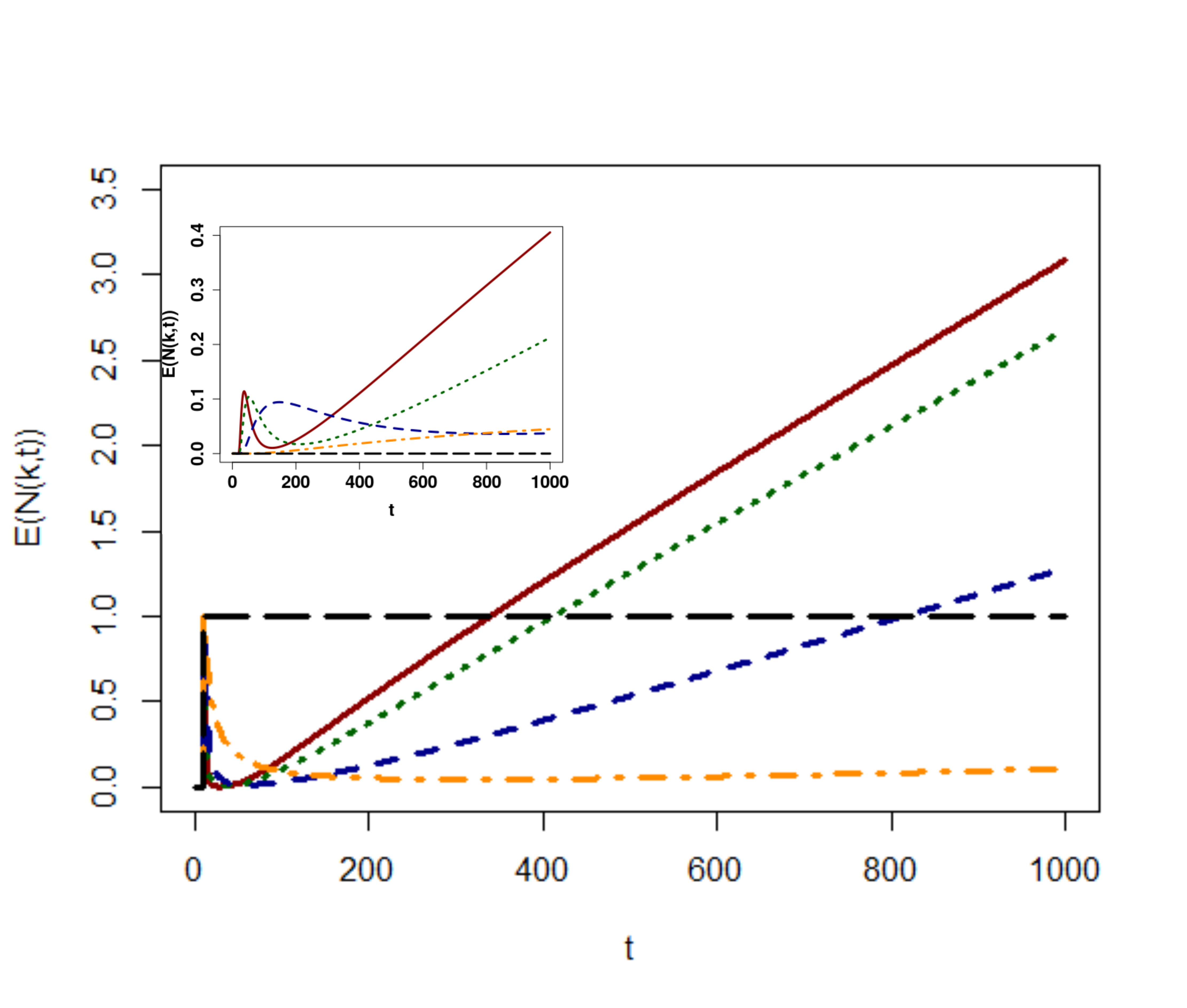}
	\end{subfigure}
\caption*{Fig.2a (left panel): The expected number of vertices with degree $k$ versus $k$ are shown, for $t=101$ and $l=10$, and probabilities of window, $p=0$ (red, solid line $\bullet$), $p=0.2$ (green, dotted line $\lozenge$), $p=0.5$ (blue, dashed line $\triangledown$), $p=0.8$ (orange, dashed dot $\blacklozenge$) and $p=1$ (black, dashed line $\circ$). \textit{Inset:} the same picture but here $t=1000$.
Fig.2b (right panel): The expected number of vertices with degree $k$ versus $t$ are shown, for $k=10$, $l=10$, and probabilities of window, $p=0$ (red, solid line), $p=0.2$ (green, dotted line), $p=0.5$ (blue, short dashed line), $p=0.8$ (orange, dash dot line) and $p=1$ (black, long dashed line). \textit{Inset:} the same picture but here $k=20$.}
	\label{figSNG2}
\end{figure}

\subsection{Asymptotic degree distribution}\label{FWS}
By using Theorem \ref{main1}, we herein illustrate the role of the UPA model parameters on the shape of the asymptotic degree distribution. In the case of the classical Barab\'asi-Albert model, the preferential attachment rule determines a scale free behavior of the degree distribution while the uniform attachment implies a uniform asymptotic distribution. To illustrate the tail behavior of the UPA model we  use (\ref{asymDegree}) and (\ref{asymDegree2}), given in Corollary \ref{remarkTeo}. 

In Fig. 3a and 3b we compare the analytical results obtained by (\ref{pkll1}), and the asymptotic approximations obtained by (\ref{asymDegree2}), only taking the first term. In the inset the comparison is with the asymptotic approximations (\ref{asymDegree2}), but taking account both the first and the second term.  More precisely, in  Fig. 3a we highlight the effect of the presence of windows by keeping its size fixed to $l=100$ while we vary the probability of the presence of windows, $p$. In Fig. 3b we highlight the effect of the size of the windows by keeping fixed $p=0.5$ and varying the size of the windows, $l$. In Fig. 3a we observe that starting from $\log   k=3$ (or $\log   k=3.5$), for $p=0.8$ (or $p=0.5$), the analytical curves fit the corresponding linear straight lines of the asymptotic approximations  (in $log-log$ scale). However, different weights for the uniform attachment change in a significant way the slope of the straight lines in the $log-log$ plot of Fig. 3a. Accounting both the first and second term in (\ref{asymDegree2}), we get a good approximation of (\ref{pkll1}) from $\log   k=2$ (or $\log   k=2.5$), for $p=0.8$ (or $p=0.5$),  (see inset of Fig. 3a). We draw the solid lines in the inset (asymptotic results taking the first and the second term) only when the asymptotic approximation starts to hold.
Results of Corollary \ref{remarkTeo}, formulae (\ref{pkl1}) and (\ref{pkll1}) show a limited role of the windows size on the asymptotic degree distribution. Only the coefficient $C_{p,l}$ in (\ref{pkll1}) depends on $l$ and this dependency is weak. Fig. 3b illustrates this property. Note that it becomes impossible to distinguish different solid lines corresponding to different values of $l$.

\begin{figure}[H]
	\centering
	\begin{subfigure}{.48\textwidth}
		\centering
		\includegraphics[width=1\linewidth]{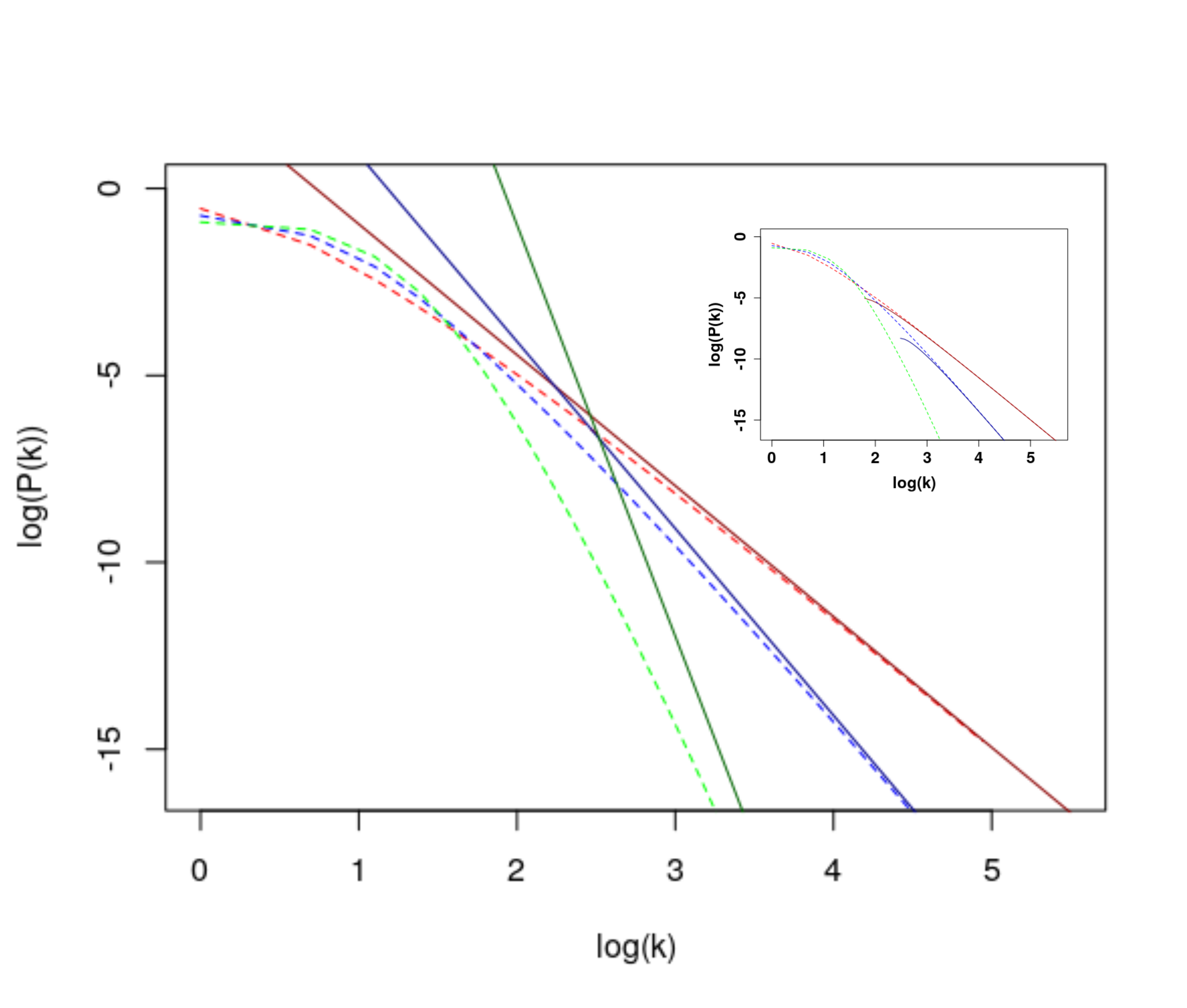}
	\end{subfigure}
	~ 
	\begin{subfigure}{.48\textwidth}
		\centering
		\includegraphics[width=1\linewidth]{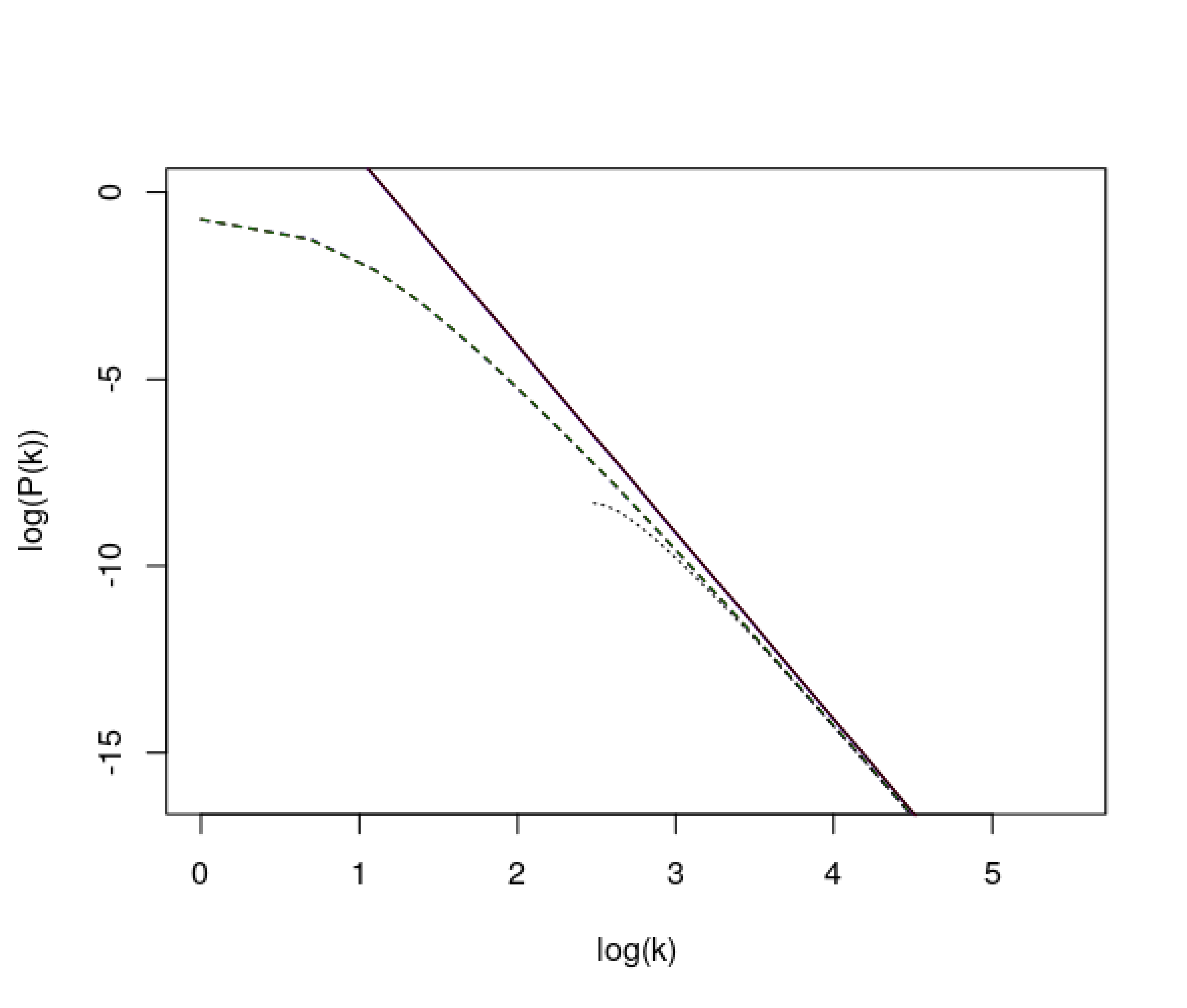}
	\end{subfigure}
\caption{Fig.3a (left panel): The  proportion of vertices with degree $k$  versus $k$ are shown in $log-log$ scale, for $l=100$ and different probabilities of window, $p=0.2$ (green), $p=0.5$ (blue) and $p=0.8$ (red). The dashed lines are analytical results, while the solid lines represent the asymptotic values taking only the first term in (\ref{asymDegree2}). \textit{Inset:} the same picture but the solid lines here refer to the asymptotic values taking the first and the second term in (\ref{asymDegree2}).
Fig.3b (right panel): $\log(P(k))$ vs. $\log(k)$.  Probability of window $p=0.5$ and different sizes of the window, $l=10$ (blue) and $l=100$ (green).
Again dashed lines are analytical results, while the solid lines represent the asymptotic values taking only the first term in (\ref{asymDegree2}).}
	\label{figSNG3}
\end{figure}

\subsection{Case of linear windows size}\label{LWS}
Our results show the impossibility to destroy the scale free behavior by introducing a second uniform attachment rule. However, it might be hypothesided that this result is determined by the fixed size of the windows. When the network increases its size, a decrement of the contribution of the uniform attachment to the global dynamics seems intuitive. For this reason we decided to investigate the dynamics of a network governed by the same rules as the one studied in the previous sections but with  windows size that grows linearly with time (and hence with the size of the network).

Unfortunately, this model cannot be studied analitically with the methodology used for the case of fixed size of the window. Therefore, we used simulations to analyze its asymptotic degree distribution. In order to choose the dimension of the network for this study, we compared analytical results with simulations for the UPA model with fixed size window. This study combined with some memory capacity restrictions, suggested to take  $n\geq100000$ as a value for which analytical results fit simulations well. 
Hence, we used $n=300000$ to analyze the asymptotic degree distribution of the network with linear size of the windows. Herein we focus on the effect of the size of the windows by keeping fixed $p=0.8$ and letting $l=\alpha n$, for different values of $\alpha$. In addition, we compare these results with the analytical results obtained by (\ref{pkll1}) for the case $l=100$, see Fig. 4. We do not observe a remarkable difference between the case of fixed and linear windows size.  Hence, we conclude that the scale free behavior is mantained also in the case of linear windows size. 

\begin{figure}[H]
		\centering
		\includegraphics[width=0.45\linewidth]{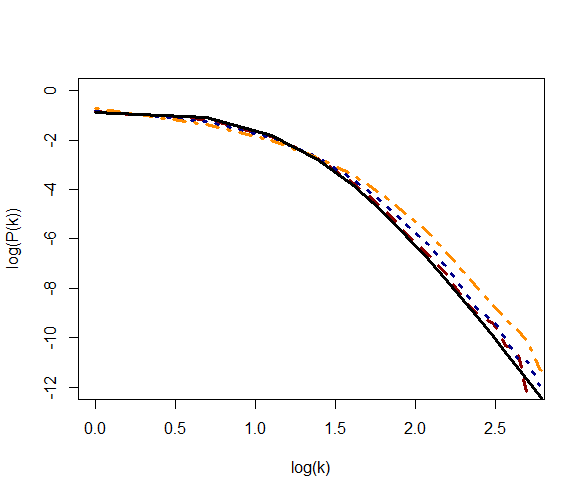}
\caption*{Fig.4:  The  proportion of vertices with degree $k$  versus $k$ are shown in $log-log$ scale, for $p=0.8$ and different sizes of the window, $l=\alpha n$, for $\alpha=0.2$ (red, long dashed line), $\alpha=0.5$ (blue, short dashed line) and $\alpha=0.8$ (orange, dash dot line). Moreover, the curve for $l=100$ (black, solid line) is also plotted taking the analytical values obtained by (\ref{pkll1}).
}\label{figSNG4}
\end{figure}

\section{Discussion and concluding remarks}
In this paper we propose a new model to describe the growth of networks in which the preferential attachment co-exists with the preferential attachment rule. Hypothesizing that the uniform attachment rule works for fixed windows size  we were able to perform an analytical study of the mean number of nodes characterized by a fixed degree. Furthermore, analytical results were proved for the degree distribution. 
From our results we can conclude that the scale free behavior does not disappear by merging preferential and uniform attachment rules. Furthermore, we show that different weights between preferential and uniform attachment rules change the exponent of the power law while the value of $l$ affects only the intercept of the $log-log$ lines and the rate of convergence to the scale free behavior. That is, depending on the value of $l$, it changes the starting value of $t$ to observe the power law behavior. This is due to the fact that when there is a window, some connections are ``wasted'' and do not contribute to the rich-get-richer mechanism.
The obtained analytical results are strongly related with the finiteness of the windows size, requested by Theorem \ref{main1}. This finiteness is indeed required. In fact, if the windows size is finite, then at each step an old node is always kept out from the window, so the degree distribution inside the window is bounded and controllable. Hence, Theorem \ref{main1} still holds when the windows shape is unconventional but finite. However, the simulation study of the model in presence of linearly growing windows size shows that the scale free behavior is preserved also in presence of unbounded windows size. This result suggest a robustness of the scale free feature that explains the incredible number of natural instances in which the phenomenon is observed. We conclude this paper conjecturing the necessity to investigate the introduction of alternative attachment rules beside the preferential attachment. It seems interesting to understand which type of attachment could destroy the asymptotic power law behavior of the degree distribution.

\section{Appendix}\label{A}
Here we complete the proof of Theorem \ref{main1}, that is, the proof of steps (2) and (3). We use the results of subsection \ref{recursiveENt}. We divide this section in two parts,  $l = 1$ and $l > 1$. \\

\textbf{PART 1. $l = 1$}\\

\begin{lemma} \label{lemmaconvk}
	In the UPA Model with fixed window size $l=1$,  ${P}(k)$, $k\geq1$, exists.
\end{lemma}

\begin{proof}
Let
\begin{equation}
S_T := \sup_{t \ge T}\frac{\bbE[{N}(k,t)]}{t} \text{ ~ and ~ } I_T := \inf_{t \ge T}\frac{\bbE[{N}(k,t)]}{t}.
\end{equation}
We divide the proof in two cases, $k=1$ and $k>1$. \\
\textit{Case $k=1$:}
We start by observing that the sequence of frequencies $\left( \bbE({N}(1,t))/t \right)_{t\geq l}$ is bounded, so we only need to prove that it is monotone.
By Lemma \ref{propA1} we obtain that
\begin{equation}
\frac{\bbE[{N}(1,t+1)]}{t+1} \lesseqqgtr \frac{\bbE[{N}(1,t)]}{t} \Leftrightarrow \frac{2(1-p)}{3-p} \lesseqqgtr \frac{\bbE[{N}(1,t)]}{t}.
\end{equation}

Let us now consider the following events:
\begin{equation}
A := \left\{ \exists T \text{ s.t. } S_T := \sup_{t \ge T}\frac{\bbE[{N}(1,t)]}{t} \le \frac{2(1-p)}{3-p} \right\} ,
\end{equation}

\begin{equation}
B := \left\{ \exists T' \text{ s.t. } I_{T'} := \inf_{t \ge T'}\frac{\bbE[{N}(1,t)]}{t} \ge \frac{2(1-p)}{3-p} \right\} 
\end{equation}
and
\begin{equation}
C := \left\{ \forall T ~ ~ S_T > \frac{2(1-p)}{3-p} > I_T \right\} .
\end{equation}
Note that the events $A$ and $B$ are mutually exclusive since their bound is determined by the same constant, while $C=(A\cup B)^c$. Therefore, $A$, $B$ and $C$ form a partition. 

Next we are going to prove monotonicity when $A$ or $B$ holds. Moreover, the union of these events  always happen. 
In fact, note that if $A$ holds
then $\bbE[{N}(1,t)]/t \le S_T \le 2(1-p)/(3-p)$ for $t \ge T$. Thus, the sequence $\left(\bbE[{N}(1,t)]/t \right)_t$ is monotone increasing. 
In the same way, if $B$ holds, 
$\bbE[{N}(1,t)]/t \ge I_{T'} \ge 2(1-p)/(3-p)$ for $t \ge T'$. Thus, the sequence $\left( \bbE[{N}(1,t)]/t \right)_t$ is monotone decreasing.

Now observe that if $C$ holds,
	for every $T\geq0$, $S_T$ is always larger than $2(1-p)/(3-p)$ and $I_T$ smaller than $2(1-p)/(3-p)$. This means that if we are in situation C,  the sequence $\left(\bbE[\overline{N}(1,t)]/t\right)_t$ should cross the value $2(1-p)/(3-p)$ infinitely many times. 

Assume that $C$ holds, then	by the definition of $\inf$ we can state that there exists $t'\geq l$ such that
	\begin{equation*}
	\label{absu1}
	 \bbE[{N}(1,t')] < t'\frac{2(1-p)}{3-p}.
	\end{equation*}
	Multiplying both sides by $\left(1-(1-p/2t')\right)$ and adding (1-p), we obtain by Lemma \ref{propA1}
	 \begin{equation*}
	 \bbE[{N}(1,t'+1)] < t' \left( 1-\frac{1-p}{2t'} \right) \frac{2(1-p)}{3-p} + (1-p)=\frac{2(1-p)}{3-p}(t'+1).
	 \end{equation*}
	
Therefore, if $\bbE[N(1,t')]/t' < 2(1-p)/(3-p)$ then we also have that $\bbE[N(1,t'+1)]/(t'+1) < 2(1-p)/(3-p)$. In other words, if at a certain point the sequence $\left( \bbE[N(1,t)]/t \right)_t$ is on the left of $2(1-p)/(3-p)$ then it will remain on the left of $2(1-p)/(3-p)$ forever, 
	so, the sequence $\left( \bbE[{N}(1,t)]/t \right)_t$ will cross the value $2(1-p)/(3-p)$ a finite number of times, which is a contradiction.




\textit{Case $k>1$:}
By Lemma \ref{propA1} we obtain that

\begin{equation}
\frac{\bbE[{N}(k,t+1)]}{t+1} \gtreqqless 
\begin{cases}
\frac{\bbE[{N}(2,t)]}{t} \Leftrightarrow \frac{1}{2-p}\left( \frac{1-p}{2} \frac{\bbE[{N}(1,t)]}{t} +p \right) \gtreqqless \frac{\bbE[{N}(2,t)]}{t},&\text{ if } k=2 \\
\frac{\bbE[{N}(k,t)]}{t} \Leftrightarrow \frac{(1-p)(k-1)}{2+(1-p)k}\left(\frac{\bbE[{N}(k-1,t)]}{t}\right)  \gtreqqless \frac{\bbE[{N}(k,t)]}{t},&\text{ if } k>2.\\
\end{cases} 
\end{equation}

Since we have already proved that $P(1)$ exists, then for each $\epsilon\in\mathbb{R}$, let us now define
\begin{equation}\label{g}
g(\epsilon) := \frac{1}{2-p}\left(\frac{1-p}{2} ({P}(1) + \epsilon) + p \right).
\end{equation}
Next we are going to write the proof when $k=2$. For $k>2$ we perform exactly the same reasoning, but replacing $g(\epsilon)$ by
\begin{equation}
h(\epsilon) := \frac{(1-p)(k-1)}{2+(1-p)k}({P}(k-1) + \epsilon).
\end{equation}


Thus, in order to prove that the sequence $\left(\bbE[N(2,t)]/t\right)_t$ converges, we are going to prove that
\begin{equation}
\label{k2limit}
\lim_{T \rightarrow \infty}(S_T-I_T) = 0 .
\end{equation}
We begin with an arbitrary time $T_1$. Then, by definition, $I_{T_1} \le S_{T_1}$, so, there exists the midpoint 
$
M_{T_1}:=(S_{T_1} + I_{T_1})/2.
$
We consider two cases $M_{T_1}\geq g(0)$ and $M_{T_1}<g(0)$, and what we are going to show is that in both cases there exists a time $T_2>T_1$ such that $S_{T_2} - I_{T_2} \le (S_{T_1} - I_{T_1})/2$. Note that using this repeatedly we obtain (\ref{k2limit}), and thus we have the convergence.

It is not difficult to verify that $g(\epsilon)$ is the equation of a line with positive slope, so it is a continuous and growing function. Then  
if $M_{T_1} \ge g(0)$, there exists a $\delta \ge 0$ such that 
$
M_{T_1} = g(\delta).
$
Moreover, by definition of $I_{T_1}$, there is a $t'>T_1$ such that 
\begin{equation}\label{Et'}
\bbE[{N}(2,t')] \le t'g(\delta),
\end{equation}
and this could happen either for a limited number or infinitely many $t'$. If this happens only a limited number of times, then there exists a $T_2 > T_1$ such that $I_{T_2} \ge g(\delta)$, so, $S_{T_2} - I_{T_2} \le \frac{1}{2}(S_{T_1} - I_{T_1})$. Otherwise, if this happens for infinitely many $t'$, then by Lemma \ref{propA1} and multiplying both sides of (\ref{Et'}) by $[1-(1-p)/t')]$ and adding $(1-p)\bbE[N(1,t)]/2t'+p$, we get that 
\begin{equation}\label{39}
\bbE[N(2,t'+1)]\leq \left(1-\frac{1-p}{t'}\right)t'g(\delta)+\frac{1-p}{2t'}\bbE[N(1,t')] + p.
\end{equation}

Since ${P}(1)$ exists, then by definition,
\begin{equation}\label{limitl1p1def}
\forall \beta > 0 ~ \exists T \text{ s.t. } \forall t \ge T ~ ~ {P}(1)-\beta < \frac{\bbE[N(1,t)]}{t} < {P}(1)+\beta .
\end{equation}
Using this we have that there exists a $T'$, such that for any $t' \ge T'$, the right side of (\ref{39}) is less than or equal 
%
\begin{align*}
 & \left(1-\frac{1-p}{t'}\right)t'g(\delta) + \frac{1-p}{2}({P}(1)+\delta) + p\\
=& t'g(\delta) - (1-p)g(\delta) + \frac{1-p}{2}({P}(1)+\delta) + p\\
=& t'g(\delta) - (1-p)g(\delta) + (2-p)g(\delta)=(t'+1)g(\delta),
\end{align*}
where the second equality follows by (\ref{g}).
Thus, we have proved  that for large values of $t'$, 
$
\bbE[{N}(2,t')]/t' \le g(\delta) ~ \Leftrightarrow ~ \bbE[{N}(2,t'+1)]/(t'+1) \le g(\delta).
$
Hence, there exists a $T_2>t'>T_1$ such that $S_{T_2} \le g(\delta)$, then 
$
S_{T_2} - I_{T_2} \le (S_{T_1} - I_{T_1})/2 .
$

In the same way we analyze the case $M_{T_1} < g(0)$. If this condition is verified, then there exists a $\delta \ge 0$ such that 
$
M_{T_1} = g(-\delta) .
$
Furthermore, by definition of $S_{T_1}$, there is a $t'$ such that 
\begin{equation}\label{nt'2}
\bbE[N(2,t')] \ge t'g(-\delta) ,
\end{equation}
what could happen for either a limited number or infinitely many $t'$. If this happens only a limited number of times, then there exists a $T_2 > T_1$ such that $S_{T_2} \le g(-\delta)$, so $S_{T_2} - I_{T_2} \le \frac{1}{2}(S_{T_1} - I_{T_1})$. Otherwise, if this happens for infinitely many $t'$, then by Proposition \ref{propA1} and multiplying both sides of (\ref{Et'}) by $[1-(1-p)/t')]$ and adding $(1-p)\bbE[N(1,t)]/2t'+p$, we get that 
\begin{equation}\label{45}
\bbE[N(2,t'+1)]\geq \left(1-\frac{1-p}{t'}\right)t'g(-\delta)+\frac{1-p}{2t'}\bbE[N(1,t')] + p.
\end{equation}
Moreover, by (\ref{limitl1p1def}), there exists a $T'$ such that for $t' \ge T'$, the right side of (\ref{45}) is greater than or equal
\begin{align*}
 & \left(1-\frac{1-p}{t'}\right)t'g(-\delta) + \frac{1-p}{2}({P}(1)-\delta) + p\\
=& t'g(-\delta) - (1-p)g(-\delta) + \frac{1-p}{2}({P}(1)-\delta) + p\\
=& t'g(-\delta) - (1-p)g(-\delta) + (2-p)g(-\delta)=(t'+1)g(-\delta),
\end{align*}
where the second equality follows by (\ref{g}).
Then, we have proved that for large values of $t'$, 
$
\bbE[N(2,t')]/t' \ge g(-\delta) ~ \Leftrightarrow ~ \bbE[N(2,t'+1)]/(t'+1) \ge g(-\delta).
$
Hence, there exists a $T_2$ such that $I_{T_2} \ge g(-\delta)$, then 
$
S_{T_2} - I_{T_2} \le (S_{T_1} - I_{T_1})/2 .
$ 
\end{proof}




\begin{lemma}\label{l1asppl}
	In the UPA model with fixed window size $l=1$, ${P}(k)$, $k\geq1$, is given by (\ref{pkl1}).
\end{lemma}

\begin{proof}
We have proved that the limit ${P}(k)$ exists for every $k \in \mathbb{N}$. Thus, by Lemma \ref{propA1} we can state that 
$
Q(k) := \lim_{t \rightarrow \infty} \left( \bbE[N(k,t+1)] - \bbE[N(k,t)] \right)
$
exists for every $k \in \mathbb{N}$. That is
\begin{equation}
\forall \epsilon > 0 ~ ~ \exists T \text{ s.t. } \forall t \ge T : Q(k)-\epsilon \le \bbE[N(k,t+1)] - \bbE[N(k,t)] \le Q(k) + \epsilon. 
\end{equation}
Using this we obtain that
\begin{align*}
(t-T)(Q(k)-\epsilon) \le& \sum_{i=0}^{t-T-1} \Big[\bbE[N(k,T+i+1)]-\bbE[N(k,T+i)]\Big] \\
=& \bbE[N(k,t)]-\bbE[N(k,T)] \le (t-T)(Q(k) + \epsilon).
\end{align*}

Furthermore, dividing by $t+1$ and taking the limit as $t \rightarrow \infty$ we get 
\begin{equation*}
Q(k)-\epsilon \le {P}(k) \le Q(k)+\epsilon ,
\end{equation*}
or, equivalently 
$
{P}(k) - \epsilon \le Q(k) \le {P}(k) + \epsilon .
$
Since $\epsilon$ is chosen arbitrarily, then 
\begin{equation}\label{Ql1}
Q(k) = {P}(k).
\end{equation}
Thus, by (\ref{Ql1}) we obtain that ${P}(1)=2(1-p)/(3-p)$ and ${P}(2)=(1-p){P}(1)/2(2-p)+p/(2-p)=(1-p)^2/(2-p)(3-p)+p/(2-p)$.
Now, for $k > 2$, by Lemma \ref{propA1} we get
\begin{equation}
\bbE[N(k,t+1)]-\bbE[N(k,t)] = -\frac{(1-p)k}{2t}\bbE[(k,t)]+\frac{(1-p)(k-1)}{2t}\bbE[N(k-1,t).
\end{equation}
Then, taking the limit as $t \rightarrow +\infty$ and using again (\ref{Ql1}), we obtain
\begin{equation}\label{ricorrenzal1}
{P}(k) = -\frac{(1-p)k}{2}{P}(k)+\frac{(1-p)(k-1)}{2}{P}(k-1)=\frac{(1-p)(k-1)}{2+(1-p)k}{P}(k-1).
\end{equation}
Now, iterating (\ref{ricorrenzal1}) we have
\begin{align*}
{P}(k) = & \frac{(1-p)^{k-2}(k-1)(k-2)\dots (k-(k-2))}{[2+(1-p)k]\dots [2+(1-p)(k-(k-3))]}{P}(2)\\
=& \frac{(k-1)!}{\left[ \frac{2}{1-p}+k \right] \cdots \left[ \frac{2}{1-p}+(k-(k-3)) \right]}{P}(2)\\
=& \frac{\Gamma(k)\Gamma\left( 3+\frac{2}{1-p} \right)}{\Gamma\left( k+1+\frac{2}{1-p} \right)}{P}(2)\\
=&{P}(2) \Big(\frac{2}{1-p}+2\Big)\Big(\frac{2}{1-p}+1\Big)B\left(k,1+\frac{2}{1-p}\right),
\end{align*}
where $\Gamma(x)$ and $B(x,y)$ are the Gamma and the Beta functions, respectively.
\end{proof}

\textbf{PART 2. $l > 1$}\\

Now we are going to prove the existence of $P(k)$. To achieve this, we use the Lemma \ref{Ewindow} and the following result.

\begin{lemma} \label{l-lemmaHk}
For $k\in\mathbb{N}$, the following limit exits,
	\begin{equation}\label{HK}
	H_k = \lim_{t \rightarrow \infty} \sum_{m=1}^{l}\mathbb{P}(M_m(t) = k) .
	\end{equation}
\end{lemma}

\begin{proof}
We start by recalling that at the starting period $t=l$ the initial graph has $l+1$ nodes, $(v_0,v_1,\dots,v_l)$, so that $deg(v_0)=l$ and $deg(v_j)=1$, $1\leq j\leq l$. That is, $\mathbb{P}(M_m(l)=1)=1$, for $1\leq m\leq l$. Furthermore, given the graph process at time $t$, we add at time $t+1$ a new vertex, $v_{t+1}$, together with an edge going out from this, so $v_{t+1}$ has degree one for all $t\geq l$, which means that for $m=l$, $\mathbb{P}(M_l(t+1)=1)=1$ for all $t\geq l$. 

Now note that since $\mathbb{P}(M_m(t+1) = 1 \mid M_{m+1}(t) \neq 1) = 0$, $t\geq l$,  then for $1\leq m < l$, 
\begin{eqnarray}\label{PM1}
\mathbb{P}(M_m(t+1) = 1) &=& \mathbb{P}(M_m(t+1) = 1 \mid M_{m+1}(t) = 1) \mathbb{P}(M_{m+1}(t) = 1) \nonumber\\
&=&I_1(t)\mathbb{P}(M_{m+1}(t) = 1),\text{ and for $k>1$},\\
\mathbb{P}(M_m(t+1) = k) &=& \mathbb{P}(M_m(t+1) = k \mid M_{m+1}(t) = k) \mathbb{P}(M_{m+1}(t) = k) \nonumber\\
&+& \mathbb{P}(M_m(t+1) = k \mid M_{m+1}(t) =k-1) \mathbb{P}(M_{m+1}(t) =k-1\nonumber\\
&=&I_k(t)\mathbb{P}(M_{m+1}(t) = k)+L_{k-1}(t)\mathbb{P}(M_{m+1}(t) =k-1),\label{PMk}
\end{eqnarray}
where,
\begin{align*}
I_k(t) &:=\mathbb{P}(M_m(t+1) = k \mid M_{m+1}(t) = k)= p\frac{l-1}{l} + (1-p)\frac{2t-k}{2t}, \text{ and}\\
L_{k-1}(t) &:=\mathbb{P}(M_m(t+1) = k \mid M_{m+1}(t) =k-1)=\frac{p}{l} + (1-p)\frac{k-1}{2t}, \quad k\in\mathbb{N}.
\end{align*}

Thus by (\ref{PM1}) and (\ref{PMk}) we can construct for $1\leq m\leq l$, the evolution of $\mathbb{P}(M_m(t) = k)$ over time. In fact if $k=1$, then  at time  $t=l$, $\mathbb{P}(M_m(l) = 1)=1$, while for $t>l$, $\mathbb{P}(M_l(t) = 1)=1$ and by (\ref{PM1}) we get that
$$
\mathbb{P}(M_m(t) = 1)=\prod_{h=1}^{l-m}I_1(t-h), \quad 1\leq m< l,
$$
with the convention that $I_1(x)=1$ when $x<l$.
In this case note that $t-h\geq l$ if $t\geq 2l-1$, and  
$\lim_{t \rightarrow \infty}I_1(t) = 1-\frac{p}{l}$.
Then, we are able to write 
	\begin{equation}
	\begin{split}
	H_1 
	=  \sum_{m=1}^{l} \left(1-\frac{p}{l}\right)^{l-m} = \frac{l}{p} - \frac{l}{p}\left(1-\frac{p}{l}\right)^l. \label{H1formula}
	\end{split}
	\end{equation}
For the case $k>1$, by (\ref{PMk}) we also note that for every $1\leq m\leq l$, and for every $t\geq l$, $\mathbb{P}(M_{m}(t) = k)$ is a finite product of functions $I_{k}(\cdot)$ and $L_{k-1}(\cdot)$ calculated at different instants. Furthermore, since
\begin{equation}
\lim_{t \rightarrow \infty} I_k(t)= 1 - \frac{p}{l} \text{ and } \lim_{t \rightarrow \infty} L_{k-1}(t) = \frac{p}{l},  
\end{equation}
so, both of these limits exist. Then, by a procedure similar to that we made to find  $H_1$, we  conclude that the limit $H_k$ also exists for $k > 1$, and
	\begin{equation}\label{Hkformula}
	H_k =\left(\frac{p}{l}\right)^{k-1}\sum_{m=1}^{l-(k-1)}\binom{l-m}{l-m-(k-1)}\left(1-\frac{p}{l}\right)^{l-m-(k-1)}.	
	\end{equation}
 
\end{proof}
Now we are ready to prove the existence of $P(k)$.
\begin{lemma} \label{l-lemmaconv1}
	In the UPA model with fixed window size $l>1$, the limit ${P}(k)$ exists for every $k \in \mathbb{N}$. 
\end{lemma}

\begin{proof}
We divide the proof in 2 cases, $k=1$ and $k>1$.  
Let $H_k(t)=\sum_{m=1}^{l}\mathbb{P}(M_{m}(t) = k)$ and 
$$
a_t=\prod\limits_{J=l}^{t} \left( 1-\frac{1-p}{2J} \right).
$$ 

Applying (\ref{theequations-l-general-1}) recursively, we get that for $k=1$
\begin{equation}
\begin{split}\label{43}
\bbE[{N}(1,t+1)] =  \bbE[{N}(1,l)]a_t  
+  \sum_{J=0}^{t-l-1}\left( 1-\frac{pH_1(t-J-1)}{l} \right) \prod\limits_{k=0}^{J}\left( 1-\frac{1-p}{2(t-k)} \right) 
+  \left(1-\frac{pH_1(t)}{l}\right).
\end{split} 
\end{equation}
Now observe the sequence $\{a_t\}_{t \in \mathbb{N}}$ is decreasing  (except when $p=1$, in this case $a_t=1, t \in \mathbb{N}$) and bounded from below by $0$. Then $\lim_{t \rightarrow +\infty}a_t$ exists and is finite. 
Hence,
\begin{equation}\label{alimit0}
\lim_{t \rightarrow +\infty} \frac{a_t}{t+1}= 0.
\end{equation}
Furthermore, by Lemma \ref{l-lemmaHk}
\begin{equation}\label{pH1limit}
\lim\limits_{t \rightarrow \infty} \frac{1}{t+1} \left(1-\frac{pH_1(t)}{l}\right) = 0, 
\end{equation}
and  for every $\epsilon > 0$ there exists a $\overline{t}$ such that 
\begin{equation}\label{Hlimitations}
H_1-\epsilon \le H_1(t) \le H_1+\epsilon, \forall t \ge \overline{t}.
\end{equation}
Note that once we chose $\epsilon$, $\overline{t}$ is fixed and we can split the summation
\begin{equation}\label{split}
\begin{split}
 \sum_{J=0}^{t-l-1}\left( 1-\frac{pH_1(t-J-1)}{l} \right) \prod\limits_{k=0}^{J}\left( 1-\frac{1-p}{2(t-k)} \right) 
= & \sum_{J=0}^{t-\overline{t}-1}\left( 1-\frac{pH_1(t-J-1)}{l} \right) \prod\limits_{k=0}^{J}\left( 1-\frac{1-p}{2(t-k)} \right) \\
+ & \sum_{J=t-\overline{t}}^{t-l-1}\left( 1-\frac{pH_1(t-J-1)}{l} \right) \prod\limits_{k=0}^{J}\left( 1-\frac{1-p}{2(t-k)} \right),
\end{split}
\end{equation}
where the second term in the split has a fixed number of elements, given by $(t-l-1)-(t-\overline{t}) = \overline{t}-l-1$. Hence,
\begin{equation}\label{48}
\lim\limits_{t \rightarrow \infty} \frac{1}{t+1} \sum_{J=t-\overline{t}}^{t-l-1}\left( 1-\frac{pH_1(t-J-1)}{l} \right) \prod\limits_{k=0}^{J}\left( 1-\frac{1-p}{2(t-k)} \right) = 0 .
\end{equation}
Let now 
\begin{equation}\label{btexpression}
b_{t} = \sum_{J=0}^{t-2}\left( \prod\limits_{k=0}^{J}\left( 1-\frac{1-p}{2(t-k)} \right) \right) .
\end{equation} 
Iterating the last equation of (\ref{l1theequations}), when $l=1$ we get
\begin{equation*}
\begin{split}
\bbE[{N}(1,t+1)] =\bbE[{N}(1,1)]a_t+ (1-p)b_t+(1-p).
\end{split} 
\end{equation*}
By (\ref{alimit0}) and  since ${P}(1)$ exists when $l=1$, then we have that $b_t/(t+1)$ converges as $t\rightarrow\infty$. Therefore,

\begin{equation}\label{splitk1approx}
\exists L \text{ s.t. } \lim\limits_{t \rightarrow \infty} \frac{1}{t+1}\sum_{J=0}^{t-\overline{t}-1} \prod\limits_{k=0}^{J}\left( 1-\frac{1-p}{2(t-k)} \right) = L. 
\end{equation}
Thus, by (\ref{Hlimitations}) and (\ref{splitk1approx}), the first term of the split in (\ref{split}) satisfies
\begin{equation}\label{100}
\begin{split}
& \left( 1-\frac{p(H_1+\epsilon)}{l} \right) L\\
\le & \lim\limits_{t \rightarrow \infty} \frac{1}{t+1} \sum_{J=0}^{t-\overline{t}-1}\left( 1-\frac{pH_1(t-J-1)}{l} \right) \prod\limits_{k=0}^{J}\left( 1-\frac{1-p}{2(t-k)} \right) \\
\le & \left( 1-\frac{p(H_1-\epsilon)}{l} \right) L.
\end{split}
\end{equation}
Since $\epsilon$ is chosen arbitrary, then the limit above exists. Hence by (\ref{43}), (\ref{alimit0}), (\ref{pH1limit}), (\ref{48}) and (\ref{100})
$$
\lim_{t\rightarrow\infty}\frac{\bbE[N(1,t+1)]}{t+1}=\left(1-\frac{pH_1}{l}\right).
$$
For the case $k > 1$, by Lemma \ref{l-lemmaHk} we know that $H_k$  exists for $k > 1$. Then, we perform the same reasoning of Lemma \ref{lemmaconvk} when $k>1$, 
but replacing $g(\epsilon)$ by
\begin{equation*}
h(\epsilon) := \left( \frac{p}{l}\left( H_{k-1}-H_k + 2\epsilon \right) + \frac{(1-p)(k-1)}{2}\left({P}(k-1)+\epsilon \right) \right) \frac{2}{2+(1-p)k}.
\end{equation*}
\end{proof}



\begin{lemma}\label{5.1}
	In the UPA model with fixed window size $l>1$, ${P}(k)$, $k\geq1$ is given by (\ref{pkll1}).	
\end{lemma}

\begin{proof}
Following the same ideas of the first lines of the proof of Lemma 
\ref{l1asppl}, we obtain
that for every $k \in \mathbb{N}$ 
	\begin{equation}\label{QQ}
	{P}(k)=\lim_{t\rightarrow +\infty} \left( \bbE[{N}(k,t+1)]-\bbE[{N}(k,t)] \right).
	\end{equation}

Now observe  that within a window of size $l$, the maximum degree of a node is $l$, then $\mathbb{P}(M_m(t)=k)=0$ for $k>l$, $1\leq m\leq l$. Therefore for $k>l+1$, by Lemma (\ref{Ewindow}) 
\begin{equation}\label{E-E}
\bbE[{N}(k,t+1)]-\bbE[{N}(k,t)]  = -\frac{(1-p)k}{2t} \bbE[{N}(k,t)] + \frac{(1-p)(k-1)}{2t}\bbE[{N}(k-1,t)].
\end{equation}
Then, taking the limit as $t\rightarrow\infty$ in (\ref{E-E}) and using (\ref{QQ}), when $k>l+1$ we get 
\begin{equation}\label{105}
{P}(k) = \frac{(1-p)(k-1)}{2+(1-p)k}{P}(k-1).
\end{equation}
Note that (\ref{105}) is equal to (\ref{ricorrenzal1}). Then, as we did in the proof of Lemma \ref{l1asppl}, iterating (\ref{105}) we obtain 
\begin{equation}
{P}(k) = \frac{\Gamma\left( k \right) \Gamma\left( l+2+
\frac{2}{1-p} \right)}{\Gamma\left( l+1 \right)
\Gamma\left( k+1+\frac{2}{1-p}\right)}{P}(l+1)=\frac{B\left(k,l+2+\frac{2}{1-p}\right)}{B\left(l+1,k+1+\frac{2}{1-p}\right)}{P}(l+1),
\end{equation}

For $1\leq k \le l+1$, we use Lemma \ref{Ewindow} and (\ref{QQ}) and obtain 
$${P}(k)=\begin{cases}
\frac{2}{(3-p)}\left(1-\frac{p}{l}\right)^l, \text{ if $k=1$}\\ 
\frac{2}{2+k(1-p)}\left(\frac{p}{l}(H_{k-1}-H_k)+\frac{(1-p)(k-1)}{2}\right){P}(k-1), \text{ if $k=2,\dots,l+1$},
\end{cases}$$
where $H_k$ is given by (\ref{H1formula}) and (\ref{Hkformula}), with  $H_k=0$ for any $k>l$.
\end{proof}



\bibliographystyle{plain}
\bibliography{BibSNG}

\end{document}